\documentclass[preprint, 10pt, english]{elsarticle}
\usepackage{amsthm}
\usepackage{amsmath}
\usepackage{latexsym, amssymb}
\usepackage{txfonts}
\usepackage{mathtools}
\usepackage{color}
\usepackage{babel}
\usepackage[all]{xy}
\usepackage[capitalise]{cleveref}

\newtheorem{thm}{Theorem}[section] 

\newtheorem{cor}[thm]{Corollary}

\newtheorem{lem}[thm]{Lemma}
\newtheorem{prop}[thm]{Proposition}

\newtheorem{ques}[thm]{Question}

\theoremstyle{definition}
\newtheorem{rem}[thm]{Remark}
\newtheorem{exmpl}[thm]{Example}

\newcommand\operA[2]{{\if!#2!\operatorname{#1}\else{\operatorname{#1}_{#2}^{\phantom{I}}}\fi}} 

\def\tr{{\operatorname{Tr}}}

\def\dim{{\operatorname{dim}}}

\newcommand{\Trace}[1][]{\if!#1!\operatorname{Tr}\else{\operatorname{Tr}_{#1}^{\phantom{I}}}\fi} 

\long\def\forget#1\forgotten{{}} %

\def\({\left(}
\def\){\right)}

\newcommand\LAY[3][]{{\begin{array}{c}\mbox{#2} \if#1!{}\else{+}\fi \\ \mbox{#3}\end{array}}}

\makeatletter

\def\ps@pprintTitle{%
 \let\@oddhead\@empty
 \let\@evenhead\@empty
 \def\@oddfoot{}%
 \let\@evenfoot\@oddfoot}

\newcommand{\bigperp}{%
  \mathop{\mathpalette\bigp@rp\relax}%
  \displaylimits
}

\newcommand{\bigp@rp}[2]{%
  \vcenter{
    \m@th\hbox{\scalebox{\ifx#1\displaystyle2.1\else1.5\fi}{$#1\perp$}}
  }%
}
\makeatother

\renewcommand{\geq}{\geqslant}
\renewcommand{\leq}{\leqslant}

\DeclareMathOperator{\abd}{ABrd}

\newif\iffurther
\furtherfalse

\journal{??}

\begin{document}
\begin{frontmatter}

\title{Asymptotic Brauer $p$-Dimension}

\author{Adam Chapman}
\address{School of Computer Science, Academic College of Tel-Aviv-Yaffo, Rabenu Yeruham St., P.O.B 8401 Yaffo, 6818211, Israel}
\ead{adam1chapman@yahoo.com}

\author{Kelly McKinnie}
\address{Department of Mathematical Sciences, University of Montana, Missoula, MT 59812, USA}
\ead{kelly.mckinnie@mso.umt.edu}

\begin{abstract}
We define and compute $\abd_p(F)$, the asymptotic Brauer $p$-dimension of a field $F$, in cases where $F$ is a rational function field or Laurent series field. $\abd_p(F)$ is defined like the Brauer $p$-dimension except it considers finite sets of Brauer classes instead of single classes. Our main result shows that for fields $F_0(\alpha_1,\dots,\alpha_n)$ and $F_0 (\!( \alpha_1)\!) \dots(\!(\alpha_n)\!)$ where $F_0$ is a perfect field of characteristic $p>0$ when $n \geq 2$ the asymptotic Brauer $p$-dimension is $n$.
We also show that it is $n-1$ when $F=F_0 (\!( \alpha_1)\!) \dots(\!(\alpha_n)\!)$ and $F_0$ is algebraically closed of characteristic not $p$. We conclude the paper with examples of pairs of cyclic algebras of odd prime degree $p$ over a field $F$ for which $\operatorname{Brd}_p(F)=2$ that share no maximal subfields despite their tensor product being non-division.
 \end{abstract}

\end{frontmatter}

\section{Introduction}

The Brauer dimension $\operatorname{Brd}(F)$ of a field $F$ is defined to be the minimal $m$ such that for any central simple algebra $A$ over $E$, $\operatorname{ind}(A)|\exp(A)^m$, where $E$ ranges over the finite field extensions of $F$.
Since $\operatorname{ind}(A)$ is also the minimal degree of a splitting field of $A$, an equivalent definition would be that every central simple algebra of exponent $e$ has a splitting field of degree dividing $e^m$.
If we consider only central simple algebras $A$ of $p$-power exponent, then the minimal $m$ is called the Brauer $p$-dimension $\operatorname{Brd}_p(F)$.
This invariant has been studied by many mathematicians. See \cite{Chipchakov:2019} and \cite{BhaskharHaase:2020} for recent results on the subject.

Inspired by Gosavi's notion of ``generalized Brauer dimension" (see \cite{gosavi2019generalized} and \cite{gosavi}),
we introduce the following generalization of the Brauer $p$-dimension: the Brauer $(p,\ell)$-dimension of a field $F$, denoted
$\operatorname{Brd}_{p,\ell}(F)$, is defined to be the minimal $m$ such that for any finite sequence of $\ell$ central simple $E$-algebras $A_1,\dots,A_\ell$ whose exponents are powers of $p$ there exists a common splitting field of degree dividing $t^m$, where $t=\max(\exp(A_1),\dots,\exp(A_\ell))$ and $E$ ranges over all the finite field extensions of $F$.
We define the asymptotic Brauer $p$-dimension of $F$, denoted $\operatorname{ABrd}_p(F)$, to be $\lim_{\ell\rightarrow \infty} \operatorname{Brd}_{p,\ell}(F)$.
Clearly $\operatorname{Brd}_p(F) \leq \operatorname{ABrd}_p(F)$.

For fields of characteristic $p$, there is an easy upper bound on the asymptotic Brauer $p$-dimension, which is the $p$-rank of the field, i.e., the number $n$ for which $[F:F^p]=p^n$.
Since the $p$-rank does not change under finite field extensions (see for example \cite[Chapter V, Section 16.6, Corollary 3]{Bourbaki}), for any finite field extension $E$ of $F$, the $p$-rank of $E$ remains $n$, and so ${_t}\operatorname{Br}(E)$ is completely trivialized by $E(\sqrt[t]{\beta_1},\dots,\sqrt[t]{\beta_n})$ where $\beta_1,\dots,\beta_n$ is a $p$-basis for $E$ and $t$ is any power of $p$ (see for example \cite[Corollary 1.2]{ParimalaSuresh:2014}). Since this is a field extension of degree $t^n$, we have $\operatorname{ABrd}_p(F) \leq n$. 
Such fields include $F_0(\alpha_1,\dots,\alpha_n)$ and $F_0(\!(\alpha_1)\!) \dots (\!(\alpha_n)\!)$ where $F_0$ is a perfect field of $\operatorname{char}(F_0)=p$.
In particular, if $F=\mathbb{F}_p(\alpha_1,\dots,\alpha_n)$  then there exists a division algebra of exponent $p$ and index $p^n$ over $F$ (see \cite[Section 5]{BhaskharHaase:2020}), hence $\operatorname{Brd}_p(F) \geq n$, but since $\operatorname{Brd}_p(F)\leq \operatorname{ABrd}_p(F)\leq n$, we have $\operatorname{Brd}_p(F)= \operatorname{ABrd}_p(F)= n$.

Our main goal in this paper is to prove the following:
\begin{thm}\label{GB}
For any perfect field $F_0$ of characteristic $p>0$ and integer $n\geq 2$, $\operatorname{ABrd}_p(F)=n$ where $F$ is either $F_0(\!(\alpha_1)\!) \dots (\!(\alpha_n)\!)$ or $F_0(\alpha_1,\dots,\alpha_n)$.
\end{thm}
Note that the Brauer $p$-dimension of $F=F_0(\!(\alpha_1)\!) \dots (\!(\alpha_n)\!)$ is known to be $n-1$ when $\dim_{\mathbb{F}_p}(F_0/\wp(F_0))<n$ and $n$ when $\dim_{\mathbb{F}_p}(F_0/\wp(F_0))\geq n$ by \cite[Proposition 5.3]{Chipchakov:2019}, so when $\dim_{\mathbb{F}_p}(F_0/\wp(F_0))<n$ (e.g., when $F_0$ is $\mathbb{F}_p$ or $\mathbb{F}_p^{sep}$) $F$ is a field whose asymptotic Brauer $p$-dimension is strictly larger than its Brauer $p$-dimension.
In Section \ref{Main} we prove this theorem by explicitly constructing a finite set of division algebras of exponent $p$ which have no common splitting field of degree $p^{n-1}$. The proof uses techniques on valued division algebras (see \cite{GilleSzamuely:2006} for background on division algebras and \cite{TignolWadsworth:2015} for their valuation theory).

We conclude the paper with another nice application of the main theorem of \cite{Morandi:1989} in the construction of examples of pairs of cyclic algebras of prime degree $p>2$ over a field $F$ of $\operatorname{char}(F)=p$ that share no maximal subfields despite their tensor product containing zero divisors. 
Such examples were first constructed in \cite{TignolWadsworth:1987} for $\operatorname{char}(F) \neq p$, then improved in \cite{JacobWadsworth:1993} to include also pairs of cyclic division algebras $A$ and $B$ of odd prime degree $p$ with no common maximal subfields, although $\operatorname{ind}(A^{\otimes i} \otimes B)=p$ for any $i$. More examples were constructed in \cite{Karpenko:1999} where over any field $F$ (without assumptions on the characteristic) any pair of algebras $A$ and $B$ with $A\otimes B$ division can be scalar extended to some field $L$ over which $\operatorname{ind}(A_L^{\otimes i} \otimes B_L)=p$ but $A_L$ and $B_L$ still share no common subfields. The example we give here, we believe, is neat enough to merit its inclusion, since it is over the field $F$ of iterated Laurent series in three variables over any field $F_0$ of characteristic $p$, as opposed to the generically constructed fields of \cite{Karpenko:1999}, and in our example, $F$ can be chosen to have Brauer $p$-dimension 2 (e.g., when $F_0$ is algebraically closed), which is lower than in Karpenko's example, and leaves the question of whether such examples exist when $\operatorname{Brd}_p(F)=1$. Our example is analogous to \cite{TignolWadsworth:1987} and not to \cite{JacobWadsworth:1993} or \cite{Karpenko:1999} because $A \otimes B$ is not division but $A^{\otimes i} \otimes B$ is for some $i \in \{2,\dots,p-1\}$.
\section{Preliminaries}

Recall that when $\operatorname{char}(F)=p$, the $p$-torsion ${_p}\operatorname{Br}(F)$ of $\operatorname{Br}(F)$ is generated by algebras of the form
$F \langle x,y: x^p-x=\alpha, y^p=\beta, yxy^{-1}=x+1 \rangle$,
for some $\alpha \in F$ and $\beta \in F^\times$, which we denote by the symbol $[\alpha,\beta)_{p,F}$, and thus these algebras are called ``symbol algebras". They are also cyclic algebras, for $F[x]/F$ is a cyclic extension.

Note that as a vector space, the algebra $[\alpha,\beta)_{p,F}$ decomposes as $\bigoplus_{i,j=0}^{p-1} F x^i y^j$, and for any $v=\sum_{i,j=0}^{p-1} c_{i,j} x^i y^j$ the trace form maps $v$ to $\tr(v)=-c_{p-1,0}$ (see \cite[Lemma 2.1]{Chapman:2021b}).
Now, if $A$ and $B$ are central simple algebras over $F$ and $r \in A$ and $t \in B$, then the trace of $r\otimes t$ in $A\otimes B$ is the product of the trace of $r$ in $A$ and the trace of $t$ in $B$.
Consequently, the subspace of trace zero elements in $[\alpha_1,\beta_1) \otimes \dots \otimes [\alpha_n,\beta_n)$ generated by $x_1,y_1,\dots,x_n,y_n$ is spanned by all $x_1^{d_1} y_1^{e_1} \dots x_n^{d_n} y_n^{e_n}$ except for the case of $d_1=\dots=d_n=p-1$ and $e_1=\dots=e_n=0$. In particular, it is of co-dimension 1, i.e., of dimension $p^{2n}-1$.

 For a given valued division algebra $D$ we follow the notation of \cite{TignolWadsworth:2015} for the residue division algebra $\overline{D}$ and the value group $\Gamma_D$ and recall freely several facts from that book.
In particular, if a field $F$ is equipped with a Henselian valuation (e.g., the field $F=F_0(\!(\alpha_1)\!)\dots(\!(\alpha_n)\!)$ with the right-to-left $(\alpha_1,\dots,\alpha_n)$-adic valuation), then the valuation extends to any central division algebra $D$ over $F$.
Another fact we will use is the fundamental inequality $[D:F] \geq |\Gamma_D/\Gamma_F| \cdot [\overline{D}:\overline{F}]$ (see \cite[Section 1.2.1]{TignolWadsworth:2015}).

\section{Equivalent conditions}

\begin{thm}[{cf. \cite[Lemma 1.1]{ParimalaSuresh:2014}}] \label{EC}
Let $p$ be a prime integer and $F$ be a field.
Then $\operatorname{Brd}_{p,\ell}(F)\leq n$ if and only if for any finite extension $E/F$, every $\ell$ algebras of exponent $p$ over $E$ have a common degree $p^n$ splitting field.
\end{thm}

\begin{proof}
The forward implication is clear. To prove the backwards implication we need to prove that if for any finite extension $E/F$, every sequence of $\ell$ algebras of exponent $p$ over $E$ have a common degree $p^n$ splitting field, then for any $m \in \mathbb{N}$ and any finite extension $E/F$, every sequence of $\ell$ algebras of exponent dividing $p^m$ over $E$ have a common degree $(p^m)^n$ splitting field.

The proof is by induction on $m$.
The case of $m=1$ coincides with the assumption.
Suppose the statement holds for $1,\dots,m-1$.
Consider a sequence of $\ell$ algebras $A_1,\dots,A_\ell$ of exponents dividing $p^m$ over a finite extension $E$ of $F$.
Then $A_1^{\otimes p},\dots,A_\ell^{\otimes p}$ are of exponents dividing $p^{m-1}$, which share a splitting field $K$ of degree $(p^{m-1})^n$ over $E$ by the induction hypothesis.
Now, the algebras $A_1 \otimes K,\dots,A_\ell \otimes K$ are of exponent dividing $p$, and thus by the assumption, have a degree $p^n$ splitting field $L$ over $K$. Hence, $L$ is a splitting field for $A_1,\dots,A_\ell$, and it is of degree $(p^{m-1})^n\cdot p^n=(p^m)^n$ over $E$.
\end{proof}

\begin{exmpl}\label{ECE}
Let $F_0$ be an algebraically closed field, let $F$ be either the function field $F_0(\alpha,\beta)$ in two variables over $F_0$ or the field $F_0(\!(\alpha)\!)(\!(\beta)\!)$. Let $E$ be a finite extension of $F$. 
For $p=2$, every three quaternion algebras over $E$ share a quadratic splitting field (see \cite{Chapman:2021}), and so $\operatorname{Brd}_{2,3}(F)=1$ by Theorem \ref{EC}.
For $p=3$, every two algebras of exponent $3$ over $E$ share a cubic splitting field, and so $\operatorname{Brd}_{3,2}(F)=1$ by Theorem \ref{EC}. For any prime $p$, it follows from \cite{Chapman:2021b} and \cite{ChapmanTignol:2019} that $\operatorname{Brd}_{p,p^2-1}(F)>1$ when $\operatorname{char}(F)=p>2$, and $\operatorname{Brd}_{p,p^2}(F)>1$ when $\operatorname{char}(F)=0$. Therefore, given the $p$-rank upper bound $\abd_p(F)\leq 2$, when $\operatorname{char}(F)=p>2$, $\operatorname{Brd}_{p,p^2-1}(F)=\operatorname{ABrd}_p(F)=2$.
\end{exmpl}

\begin{cor}
For any field $F$, $\operatorname{Brd}_2(F)=1$ if and only if $\operatorname{Brd}_{2,2}(F)=1$.
\end{cor}

\begin{proof}
Suppose $\operatorname{Brd}_2(F)=1$.
Then for any finite extension $E$ of $F$, any biquaternion algebra $A\otimes B$ over $E$ contains zero divisors, and therefore $A$ and $B$ share a quadratic splitting field by Albert (see \cite{Albert:1972}). Therefore $\operatorname{Brd}_{2,2}(F)=1$.
The opposite direction is clear.
\end{proof}

\begin{exmpl}
When $F_0$ is algebraically closed, the field of iterated Laurent series $F=F_0(\!(\alpha)\!)(\!(\beta)\!)(\!(\gamma)\!)$ in three variables over $F_0$ has $\operatorname{Brd}_2(F)=1$ by \cite{Chipchakov:2019}, and therefore also $\operatorname{Brd}_{2,2}(F)=1$.
However, $\operatorname{Brd}_{2,3}(F)> 1$ because if it were $1$, by \cite{ChapmanDolphinLeep} and \cite{Becher} we would have $I_q^3 F=0$, but $I_q^3 F \neq 0$.
Since for every three quaternion algebras over a finite extension $E$ of $F$, the first two have a common quadratic splitting field $K$ and the last has a quadratic splitting field $L$, the three algebras have the compositum $K.L$ as the common bi-quadratic splitting field, and thus $\operatorname{Brd}_{2,3}(F)=2$.
\end{exmpl}

\begin{ques}
Does $\operatorname{Brd}_p(F)=1$ imply $\operatorname{Brd}_{p,2}(F)=1$ for every prime $p$?
\end{ques}
It is known that Albert's theorem \cite{Albert:1972} does not extend to algebras of exponent $p>2$. However, the existing examples in the literature (see \cite{TignolWadsworth:1987}, \cite{JacobWadsworth:1993}, \cite{Karpenko:1999} and Section \ref{charp}) are constructed over fields $F$ with $\operatorname{Brd}_p(F) \neq 1$, and therefore do not resolve the question.
\section{Iterated Laurent series over perfect fields}\label{Main}

In this section we prove Theorem \ref{GB}.
We focus on $F=F_0(\!(\alpha_1)\!)\dots(\!(\alpha_n)\!)$ for convenience, but all the algebras under discussion are defined over $F_0(\alpha_1,\dots,\alpha_n)$, and one obtains the inexistence of a common degree $p^{n-1}$ splitting field from the inexistence of such a splitting field under restriction of these algebras to $F_0(\!(\alpha_1)\!)\dots(\!(\alpha_n)\!)$.
In what follows, $F=F_0(\!(\alpha_1)\!)\dots(\!(\alpha_n)\!)$ where $F_0$ is a field of characteristic $p>0$ and $n \geq 2$. Note that for $n=1$, all the central simple algebras over $F$ of exponent dividing $p^\ell$ are split by $F[\sqrt[p^\ell]{\alpha_1}]$ and the Asymptotic Brauer $p$-dimension is 1 (unless $H_p^1(F_0)=0$, in which case it is 0).

We cite the following useful result from \cite[Theorem 1]{Morandi:1989}: If $F$ is a Henselian valued field, $D$ and $E$ are division algebras over $F$ such that $D$ is defectless, $\overline{D} \otimes \overline{E}$ is a division algebra and $\Gamma_D \cap \Gamma_E=\Gamma_F$, then $D\otimes E$ is a division algebra as well.

\begin{lem}\label{Shift}
For $n \geq 3$, pick $i \in \{1,\dots,n-1\}$ and write $\beta_1,\dots,\beta_{n-2}$ for $\alpha_1,\dots,\widehat{\alpha_i},\dots,\alpha_{n-1}$. Then the algebra
$$A_i=[\alpha_n^{-1},\beta_{n-2})_{p,F} \otimes [\beta_{n-2}^{-1},\beta_{n-3})_{p,F} \otimes \dots \otimes [\beta_2^{-1},\beta_1)_{p,F} \otimes [\beta_1^{-1},\alpha_i)_{p,F}$$
is a division algebra.
\end{lem}

\begin{proof}
The case of $i=1$ was covered in \cite{Chapman:2020}.
Suppose $i>1$.
Write $D=[\alpha_n^{-1},\beta_{n-2})_{p,F} \otimes \dots \otimes [\beta_2^{-1},\beta_1)_{p,F}$ and $E=[\beta_1^{-1},\alpha_i)_{p,F}$.
The algebra $D$ has degree $p^{n-2}$ and is division by \cite{Chapman:2020}. The algebra $E$ is division because the values of $\beta_1^{-1}$ and $\alpha_i$ are $\mathbb{F}_p$-independent in $\Gamma_F/p\Gamma_F$ w.r.t the standard valuation on $F$.
Consider the $(\alpha_2,\dots,\alpha_n)$-adic valuation on $F$.
The algebra $D$ is defectless, because $|\Gamma_D/\Gamma_F|=p^{2n-5}$ and $\overline{D}=\overline{F}[y : y^p=\alpha_1]$. Moreover, $\Gamma_D \cap \Gamma_E=\Gamma_F$ and $\overline{D} \otimes \overline{E}=\overline{F}[x,y : x^p-x=\alpha_1^{-1}=y^p]$ is the field generated over $\overline{F}=F_0(\!(\alpha_1)\!)$ by $z=x-y$ satisfying $z^{p^2}-z^p=\alpha_1^{-1}$. By \cite{Morandi:1989} $D\otimes E$ is division.
\end{proof}

Since this lemma referred to the cases of $n\geq 3$, we can add that for $n=2$, we can define the algebra $A_1=[\alpha_2^{-1},\alpha_1)_{p,F}$, which is division because the values of $\alpha_2^{-1}$ and $\alpha_1$ are $\mathbb{F}_p$-independent in $\Gamma_F/p\Gamma_F$ w.r.t the standard valuation on $F$.

\begin{prop}
Consider from now on the $(\alpha_1,\dots,\alpha_n)$-adic valuation on $F$, denoted by $\mathfrak{v}$.
For each $i \in \{1,\dots,n-1\}$, the algebra $A_i$ is totally ramified and its value group $\Gamma_{A_i}$ is 
$$\frac{1}{p^2} \mathbb{Z} \times \dots \times \frac{1}{p^2} \mathbb{Z} \times \underbrace{\frac{1}{p} \mathbb{Z}}_{i\text{th place}} \times \frac{1}{p^2} \mathbb{Z} \times \dots \times \frac{1}{p^2} \mathbb{Z} \times \frac{1}{p} \mathbb{Z}.$$
As a result, $\bigcap_{i=1}^{n-1} \Gamma_{A_i}=\frac{1}{p} \mathbb{Z} \times \dots \times \frac{1}{p} \mathbb{Z}$.
\end{prop}

\begin{proof}
Since $A_i$ is a division algebra and $F$ is Henselian, the valuation extends to $A_i$.
By the fundamental inequality, $[A_i:F] \geq |\Gamma_{A_i}/\Gamma_F| \cdot [\overline{A_i}:\overline{F}]$, it is enough to show that $\Gamma_{A_i} \supseteq \frac{1}{p^2} \mathbb{Z} \times \dots \times \frac{1}{p^2} \mathbb{Z} \times \underbrace{\frac{1}{p} \mathbb{Z}}_{i\text{th place}} \times \frac{1}{p^2} \mathbb{Z} \times \dots \times \frac{1}{p^2} \mathbb{Z} \times \frac{1}{p} \mathbb{Z}$, and then all the inequalities would become equalities, and $\overline{A_i}=\overline{F}$.
For each $k \in \{1,\dots,n-1\} \setminus \{i\}$, $A_i$ contains the subfield $F[x,y: x^p-x=\alpha_k^{-1}, y^p=\alpha_k]$, which is in fact a simple extension of $F$ generated by $t=x-y^{-1}$ satisfying $t^p-t=y^{-1}$, and thus $\mathfrak{v}(t)=\frac{1}{p}\mathfrak{v}(y^{-1})=\frac{1}{p^2}\mathfrak{v}(\alpha_k^{-1})$.
In addition, $A_i$ contains the field $F[x : x^p-x=\alpha_n^{-1}]$, and here $\mathfrak{v}(x)=\frac{1}{p}\mathfrak{v}(\alpha_n^{-1})$, and the field $F[y : y^p=\alpha_i]$, and here $\mathfrak{v}(y)=\frac{1}{p}\mathfrak{v}(\alpha_i)$.
\end{proof}

\begin{rem}\label{More}
In addition to the algebras constructed in Lemma \ref{Shift}, the following algebras are also division algebras, by the same reasoning:
$$B_{d_1,\dots,d_n}=[\prod_{k=1}^n \alpha_k^{-d_k},\alpha_{n-1})_{p,F} \otimes [\alpha_{n-1}^{-1},\alpha_{n-2})_{p,F} \otimes \dots \otimes [\alpha_2^{-1},\alpha_1)_{p,F},$$
where $d_1,\dots,d_{n-1} \in \{0,\dots,p-1\}$ and $d_n \in \{1,\dots,p-1\}$, and 
$$B_{d_1,\dots,d_{n-1},0}=[\prod_{k=1}^{n-1} \alpha_k^{-d_k},\alpha_{n-2})_{p,F} \otimes [\alpha_{n-2}^{-1},\alpha_{n-3})_{p,F} \otimes \dots \otimes [\alpha_2^{-1},\alpha_1)_{p,F} \otimes [\alpha_1^{-1},\alpha_n)_{p,F},$$
where $d_1,\dots,d_{n-2} \in \{0,\dots,p-1\}$ and $d_{n-1} \in \{1,\dots,p-1\}$.
Note that $B_{0,\dots,0,1}=A_1$.
\end{rem}

\begin{thm}
Assuming $p$ is a prime integer, $n\geq 2$ and $(n,p) \neq (2,2)$, the $n-2+p^n-p^{n-2}$ algebras $\{A_i,B_{d_1,\dots,d_n} : i \in \{2,\dots,n-1\},(d_1,\dots,d_n)\in \{0,\dots,p-1\}^{\times n}, (d_{n-1},d_n)\neq (0,0)\}$ share no splitting field of degree $p^{n-1}$.
\end{thm}

\begin{proof}
Suppose $K$ is a common degree $p^{n-1}$ splitting field.
Since all these division algebras are of degree $p^{n-1}$, $K$ embeds into all of them.
Since $K$ embeds into each $A_i$, we have $\Gamma_K \subseteq \bigcap_{i=1}^{n-1} \Gamma_{A_i}=\frac{1}{p} \mathbb{Z} \times \dots \times \frac{1}{p} \mathbb{Z}$, considering the right-to-left $(\alpha_1,\dots,\alpha_n)$-adic valuation $\mathfrak{v}$ on $F$ extended to each $A_i$.
The subspace $V$ of $K$ of elements of trace 0 is of dimension at least $p^{n-1}-1$, and since $K$ is totally ramified over $F$, the set of classes of $\mathfrak{v}(V)$ modulo $\Gamma_F$ is of cardinality at least $p^{n-1}-1$.
As we noted, $\mathfrak{v}(V)$ modulo $\Gamma_F$ is a subset of $\frac{1}{p} \mathbb{Z} \times \dots \times \frac{1}{p} \mathbb{Z}/\Gamma_F$.
Now, $\mathfrak{v}(V)$ must be a subset of the set of values of trace zero elements in $B_{d_1,\dots,d_n}$, which means that $\mathfrak{v}(V)$ mod $\Gamma_F$ is a subset of 
$$(\frac{1}{p} \mathbb{Z} \times \dots \times \frac{1}{p} \mathbb{Z}/\Gamma_F)\setminus \{(\frac{d_1}{p},\dots,\frac{d_n}{p}):(d_1,\dots,d_n) \in \{0,\dots,p-1\}^{\times n}, (d_{n-1},d_n)\neq (0,0)\}.$$
However, the latter is of cardinality  $p^n-(p-1)(p^{n-1}+p^{n-2})=p^{n-2}$, and $p^{n-2}<p^{n-1}-1$, contradiction.
\end{proof}
This theorem shows that $\abd_p(F)>n-1$ for these fields in all cases but $n=p=2$, a case that was covered in \cite{Chapman:2021}.
Since for these fields, $\abd_p(F)\leq n$, the main theorem (Theorem \ref{GB}) follows.
\begin{rem}
Theorem 3.4 actually proves that $\operatorname{Brd}_{p,n-2+p^n-p^{n-2}}(F)=n$ when $F_0$ is perfect of $\operatorname{char}(F_0)=p>0$, $F$ is either $F_0(\alpha_1,\dots,\alpha_n)$ or $F_0(\!(\alpha_1)\!)\dots(\!(\alpha_n)\!)$ and $n\geq 2$ with $(n,p)\neq (2,2)$.
For $n=p=2$, $n-2+p^n-p^{n-2}=3$, and it is not always true that $\operatorname{Brd}_{2,3}(F)=2$. For $n=p=2$, when $F_0=\mathbb{F}_2^{sep}$, we have $\operatorname{Brd}_{2,3}(F)=1$ (see Example \ref{ECE}).
However, for any perfect $F_0$, there exist four quaternion algebras over $F$ with no common quadratic splitting field, and therefore $\operatorname{Brd}_{2,4}(F)=2$.
\end{rem}

\section{Iterated Laurent Series over  Fields of Characteristic not $p$}
Let $p$ be a prime integer.
Recall that over fields $F$ of $\operatorname{char}(F)\neq p$ containing primitive $p$th roots of unity $\rho$, ${_p\operatorname{Br}}(F)$ is generated by symbol algebras of degree $p$ (\cite{MS}), which take the form $(\alpha,\beta)_{p,F}=F\langle i,j : i^p=\alpha, j^p=\beta, j i=\rho i j \rangle$. Let $F_0$ be an algebraically closed field of characteristic not $p$. Let $n$ be a natural number $\geq 2$, and $F=F_0(\!(\alpha_1)\!)\dots(\!(\alpha_n)\!)$ the field of iterated Laurent series in $n$ variables over $F_0$.

The Brauer $p$-dimension of $F$ in this case is known to be $\left\lfloor \frac{n}{2} \right\rfloor$ (see \cite{Chipchakov:2019}).
When $n=2$, the group ${_{p}}Br(F)$ is the cyclic group $\langle (\alpha_1,\alpha_2)_{p,F} \rangle$ of order $p$, and thus $\operatorname{ABrd}_p(F)=\operatorname{Brd}_p(F)$.
For $n \geq 3$, the asymptotic Brauer $p$-dimension may be greater.
Since the set $S$ of classes of index $p^{\left\lfloor \frac{n}{2} \right\rfloor}$ in ${_p\operatorname{Br}}(F)$ is finite, it is enough to consider the common splitting fields of the classes in $S$.

\begin{thm}
Suppose $n\geq 3$. Then $\operatorname{ABrd}_p(F)=n-1$.
\end{thm}

\begin{proof}
Let $K/F$ be a field extension of degree dividing $p^{n-2}$.
The natural valuation on $F$ extends to $K$ for $F$ is henselian.
We claim that there exist $\gamma$ and $\delta$ in $F^\times$ whose values are $\mathbb{F}_p$-independent in $\Gamma_K/p\Gamma_K$. Because otherwise, the dimension of $(\Gamma_K \cap \frac{1}{p}\Gamma_F)/p(\Gamma_K \cap \frac{1}{p}\Gamma_F)$ over $\mathbb{F}_p$ would be at least $n-1$, and then $|\Gamma_K/\Gamma_F|$ would be at least $p^{n-1}$, contradictory to the fact that $|\Gamma_K/\Gamma_F|\leq p^{n-2}$.
Therefore, $(\gamma,\delta)_{p,K}$ is a division algebra.
In particular, $K$ does not split $(\gamma,\delta)_{p,F}$, which means there is no common splitting field of ${_p}\operatorname{Br}(F)$ of degree dividing $p^{n-2}$. Hence, $\operatorname{ABrd}_p(F)\geq n-1$.

Now, for any finite field extension $E$ of $F$, take $\beta_1,\dots,\beta_n$ to be elements whose values generate $\Gamma_E$, and then the field extension $E[\sqrt[p]{\beta_1},\dots,\sqrt[p]{\beta_{n-1}}]$ splits the entire group ${_{p}Br}(E)$.
Consequently, $\operatorname{ABrd}_p(F)=n-1$.
\end{proof}

Since for $n\geq 3$, we have $n-1>\left\lfloor \frac{n}{2} \right\rfloor$, this provides a family of fields where $\operatorname{ABrd}_p(F)>\operatorname{Brd}_p(F)$ with a gap as large as we want.

\section{Further open questions}

Clearly fields of infinite Brauer $p$-dimension exits, such as the field of functions with $\aleph_0$ variables $F_0(\alpha_1,\dots)$. For this field, the asymptotic Brauer $p$-dimension is also infinite.

\begin{ques}
Are there fields $F$ with $\operatorname{ABrd}_p(F)=\infty>\operatorname{Brd}_p(F)$?
\end{ques}

A natural candidate would be the field $F=F_0(\alpha,\beta)$ with $F_0$ algebraically closed of characteristic 0.

\begin{ques}
What is $\operatorname{ABrd}_p(\mathbb{C}(\alpha,\beta))$?
\end{ques}

\section{Cyclic algebras without common maximal subfields}\label{charp}

\begin{prop}[{cf. \cite[Theorem 1.1 (i)]{JacobWadsworth:1993}}] \label{JWP}
Let $F_0$ be a field of odd characteristic $p$ and $F=F_0(\!(t)\!)$ the field of Laurent series over $F$. 
\begin{enumerate}
\item Write $A=[a,t)_{p,F}$ and $B=[c,dt)_{p,F}$.
Then $A\otimes B$ is division if and only if $[c,d)_{p,F_0[\wp^{-1}(a+c)]}$ is division.
\item Write $A=[t^{-1},a)_{p,F}$ and $B=[d+t^{-1},c)_{p,F}$.
Then $A\otimes B$ is division if and only if $[d,c)_{p,F_0[\sqrt[p]{ac}]}$ is division.
\end{enumerate}
\end{prop}

\begin{proof}
Write $A=[a,t)_{p,F}$ and $B=[c,dt)_{p,F}$.
The tensor product $A \otimes B$ is isomorphic to $D \otimes E$ where $D=[a+c,t)_{p,F}$ and $E=[c,d)_{p,F}$.
Consider the $t$-adic valuation on $F$.
Since $\overline{D} \otimes \overline{E}=[c,d)_{p,F_0[\wp^{-1}(a+c)]}$, if it is division, then by \cite{Morandi:1989}, $D\otimes E$ is division.
In the opposite direction, if $[c,d)_{p,F_0[\wp^{-1}(a+c)]}$ is not division, then since it is isomorphic to a subalgebra of $D\otimes E$, the latter is not division either.

Now, write $A=[t^{-1},a)_{p,F}$ and $B=[d+t^{-1},c)_{p,F}$.
The tensor product $A \otimes B$ is isomorphic to $D \otimes E$ where $D=[t^{-1},ac)_{p,F}$ and $E=[d,c)_{p,F}$.
Consider the $t$-adic valuation on $F$.
Since $\overline{D} \otimes \overline{E}=[d,c)_{p,F_0[\sqrt[p]{ac}]}$, if it is division, then by \cite{Morandi:1989}, $D\otimes E$ is division.
In the opposite direction, if $[d,c)_{p,F_0[\sqrt[p]{ac}]}$ is not division, then since it is isomorphic to a subalgebra of $D\otimes E$, the latter is not division either.
\end{proof}

\begin{lem}\label{Tignol}
Given a prime integer $p$ and a field $k$ of $\operatorname{char}(k)=p$, let $F_0=k(\!(d)\!)(\!(c)\!)$.
\begin{enumerate}
\item The field $K=F_0[x : x^p-x=d^{-1}]$ is not a subfield of the symbol division algebra $D=[c^{-1},d^{-1})_{p,F_0}$.
\item The field $K=F_0[y : y^p=d]$ is not a subfield of the symbol division algebra $D=[d^{-1},c)_{p,F_0}$.
\end{enumerate}
\end{lem}

\begin{proof}
Write $K=F_0[x : x^p-x=d^{-1}]$ and $D=[c^{-1},d^{-1})_{p,F_0}$.
Consider the right-to-left $(d,c)$-adic valuation $\mathfrak{v}$.
Recall that for any division algebra $A$ over $F$, $w(A)$ denotes the minimum of $\mathfrak{v}(\tr(z))-\mathfrak{v}(z)$ where $z$ ranges over all elements of $A$. By \cite[Theorem 2.6]{Tignol:1992} $w(\operatorname{Trd}_{D/F}) = (0,\frac{p-1}{p})$ for the division algebra $D$, and by \cite[Theorem 1.8]{Tignol:1992} $w(\operatorname{Tr}_{K/F_0}) = (\frac{p-1}{p},0)$ for the field $K$. But by definition, $w(\operatorname{Trd}_{D/F}) \leq w(\operatorname{Tr}_{K/F_0})$ if $K$ is a subfield of $D$, whereas $(0,\frac{p-1}{p}) > (\frac{p-1}{p},0)$. 

Now, write $K=F_0[y : y^p=d]$ and $D=[d^{-1},c)_{p,F_0}$.
Since $K=k(\!(y)\!)(\!(c)\!)$ and $D=[d^{-1},c)_{p,F_0}=[y^{-p},c)_{p,F_0}$, we have $D \otimes K=[y^{-1},c)_{p,k(\!(y)\!)(\!(c)\!)}$, which is a division algebra because the values of $y^{-1}$ and $c$ are $\mathbb{F}_p$-independent in $\Gamma_K/p\Gamma_K$.
\end{proof}

\begin{exmpl}
Suppose $p$ is an odd prime. Let $k$ be a field of $\operatorname{char}(k)=p$, $F_0=k(\!(d)\!)(\!(c)\!)$, $F=F_0(\!(t)\!)$, $A=[d^{-1}-c^{-1},t)_{p,F}$, $B=[2d^{-1}-2c^{-1},t)_{p,F}$ and $C=[c^{-1},d^{-1}t)_{p,F}$.
Since $[c^{-1},d^{-1})_{p,K}$ is division by Lemma \ref{Tignol} where $K=F[\wp^{-1}(d^{-1})]=F[\wp^{-1}(d^{-1}-c^{-1}+c^{-1})]$, the algebra $A\otimes C$ is division by Proposition \ref{JWP}, and thus $A$ and $C$ do not share any maximal subfield.
The algebras $B$ and $C$ do not share any maximal subfield either, because $B$ is Brauer equivalent to $A \otimes A$.
However, $[c^{-1},d^{-1})_{p,L}$ is split when $L=F[\wp^{-1}(2d^{-1}-c^{-1})]=F[\wp^{-1}(2d^{-1}-2c^{-1}+c^{-1})]$ because $[c^{-1},d^{-1})_{p,L}=[c^{-1}-2d^{-1},d^{-1})_{p,L}=[2d^{-1}-c^{-1},d)_{p,L}$.
Therefore, $B$ and $C$ form a pair of cyclic degree $p$ algebras that share no maximal subfield despite their tensor product being non-division by Proposition \ref{JWP}.

Now, write $A=[t^{-1},dc^{-1})_{p,F}$, $B=[t^{-1},d^2c^{-2})_{p,F}$ and $C=[d^{-1}+t^{-1},c)_{p,F}$.
Since $[d^{-1},c)_{p,K}$ is division where $K=F[\sqrt[p]{d}]=F[\sqrt[p]{dc^{-1}\cdot c}]$ by Lemma \ref{Tignol}, the algebra $A\otimes C$ is division by Proposition \ref{JWP}, and thus $A$ and $C$ do not share any maximal subfield.
The algebras $B$ and $C$ do not share any maximal subfield either, because $B$ is Brauer equivalent to $A \otimes A$.
However, $[d^{-1},c)_{p,L}$ is split when $L=F[\sqrt[p]{d^2c^{-1}}]=F[\sqrt[p]{d^2c^{-2}\cdot c}])]$ because $[d^{-1},c)_{p,L}=[d^{-1},cd^{-2})_{p,L}=[-d^{-1},d^2c^{-1})_{p,L}$.
Therefore, $B$ and $C$ form a pair of cyclic degree $p$ algebras that share no maximal subfield despite their tensor product being non-division by Proposition \ref{JWP}.
\end{exmpl}

When $F_0$ is algebraically closed, the field $F$ from the previous example has Brauer $p$-dimension 2, like in the analogous examples from \cite{TignolWadsworth:1987}.

%

\bibliographystyle{abbrv}
\bibliography{bibfile}

\end{document}